\documentclass[10pt]{article}
\usepackage{epsfig}
\usepackage{amssymb,amsmath,amsthm,url}
\usepackage{multirow}
\usepackage{booktabs}
\usepackage{natbib}
\usepackage{hyperref}
\usepackage{setspace}
\usepackage{dsfont}

\usepackage{times,amsfonts,eucal}
\usepackage{graphicx,ifthen}
\usepackage{wrapfig}
\usepackage{latexsym}
\usepackage{color}
\usepackage{mathrsfs}
\usepackage{bm}
\usepackage{bbm}
\usepackage{srctex}

\usepackage{mathtools}
\DeclarePairedDelimiter{\floor}{\lfloor}{\rfloor}
\DeclarePairedDelimiter\ceil{ \lceil}{ \rceil}
\theoremstyle{plain}
\newtheorem{theorem}{Theorem}[section]
\newtheorem{corollary}[theorem]{Corollary}
\newtheorem{lemma}[theorem]{Lemma}
\newtheorem{proposition}[theorem]{Proposition}

\theoremstyle{definition}
\newtheorem{definition}[theorem]{Definition}

\theoremstyle{remark}

\numberwithin{equation}{section}

\usepackage{bbm}

\newcommand{\N}{\mathbbm{N}}
\newcommand{\R}{\mathbbm{R}}

\newcommand{\Z}{\mathbbm{Z}}

\newcommand{\p}{\mathbbm{P}}

\newcommand{\cA}{\mathcal{A}}

\newcommand{\cG}{\mathcal{G}}
\newcommand{\cF}{\mathcal{F}}
\newcommand{\cH}{\mathcal{H}}

\newcommand{\cN}{\mathcal{N}}

\newcommand{\cO}{\mathcal{O}}

\newcommand{\E}[1]{\mathbbm{E}\left [ \, #1 \, \right ]}

\renewcommand{\epsilon}{\varepsilon}
\renewcommand{\phi}{\varphi}
\newcommand{\pspace}{(\Omega,\cA,\p)}

\newcommand{\norm}[1]{\left\lVert #1 \right\rVert}

\newcommand{\1}[1]{\,\mathbbm{1}\! \left\{ #1 \right\} }
\allowdisplaybreaks
\begin{document}

\title{A Bernstein Inequality For Exponentially Growing Graphs\footnote{This research was supported by the Fraunhofer ITWM, 67663 Kaiserslautern, Germany which is part of the Fraunhofer Gesellschaft zur F{\"o}rderung der angewandten Forschung e.V. The author thanks Hannes Christiansen for proofreading parts of the article.}
}
\author{Johannes T. N. Krebs\footnote{Department of Mathematics, University of Kaiserslautern, 67653 Kaiserslautern, Germany, email: \tt{krebs@mathematik.uni-kl.de} }
}
\date{\today}
\maketitle

\vspace{1cm}

\begin{abstract}
\setlength{\baselineskip}{1.8em}
In this article we present a Bernstein inequality for sums of random variables which are defined on a graphical network whose nodes grow at an exponential rate. The inequality can be used to derive concentration inequalities in highly-connected networks. It can be useful to obtain consistency properties for nonparametric estimators of conditional expectation functions which are derived from such networks. \medskip
\end{abstract}

\vspace{1cm}


Inequalities of the Bernstein type are an important tool for the asymptotic analysis in probability theory and statistics. The original inequality derived by \cite{bernstein1927} gives bounds on $\p( |S_n| > \epsilon)$, where $S_n=\sum_{k=1}^n Z_k$ for bounded random variables $Z_1,\ldots,Z_n$ which are i.i.d. and have expectation zero. There are various versions of Bernstein's inequality, e.g., \cite{hoeffding1963probability}. In particular, generalizations to different kinds of stochastic processes have gained importance: \cite{carbon1983}, \cite{collomb1984proprietes}, \cite{bryc1996large} and \cite{merlevede2009} provide extensions to times series $\{ Z_t: t\in\Z\}$ which are weakly dependent. \cite{frankeBernstein} give a further generalization to strong mixing random fields $\{Z_s: s\in\Z^N\}$ which are defined on the regular lattice $\Z^N$ for some lattice dimension $N \in\N_+$. The corresponding definitions of dependence are given in \cite{doukhan1994mixing} and in \cite{bradley2005basicMixing}.\\
Bernstein inequalities in particular find their applications when deriving large deviation results in nonparametric regression and density estimation, compare \cite{gyorfi1989nonparametric} and \cite{gyorfi}.\\
In this article we derive a new Bernstein inequality which adapts to highly-connected networks where the number of nodes grows at an exponential rate. A well-known example for such a graph is the internet map which tries to represent the internet with visual graphics. Another application may be nested simulations which are used in insurance mathematics to simulate the outcome of an insurance contract. Based on this new Bernstein inequality, we derive a concentration inequality which ensures that in simulations the nonparametric regression or density function estimator is consistent. It turns out that we need a somewhat stricter decay in the $\alpha$-mixing coefficients than it is usually assumed in the case for time series. Due to the special geometric structure of the underlying data, many technical aspects in the proofs of these new inequalities are much more involved than it is the case for time series data or for data which is defined on a lattice.\\
This paper is organized as follows: we give the motivation and the definitions in Section~\ref{Introduction}. Section~\ref{Bernstein} contains the new Bernstein inequality and concentration inequalities for exponentially growing graphs, it is the main part of this article. The Appendix~\ref{Appendix} contains a useful result of \cite{davydov1968convergence}.

\section{Introduction}~\label{Introduction}

In this section we consider a general graph $G=(V,E)$ with a countable set of nodes $V$ and a set of edges $E$. We define the natural metric on $G$ as the minimal number of edges between two nodes
\begin{align*}
		&d_G: V \times V \rightarrow \N, \\
		&\qquad \qquad (v,w) \mapsto \inf \big \{l \in \N \text{ such that there are } (v_0,v_1),...,(v_{l-1},v_l) \in E \text{ with } \, v_0 = v, v_l = w \big\}.
\end{align*}
The metric $d_G$ is extended to sets $I,J \subseteq V$ in the usual way: $d_G(I,J) = \inf\{ d_G(v,w): v\in I, w\in J\}$. We denote by $\cN(v)$ the set of neighbors of $v$ w.r.t.\ $G$ for a node $v$ of a graph $G=(V,E)$. Furthermore, we assume that there is a probability space $\pspace$ which is endowed with a real-valued random field $Z$. The latter is indexed by the set of nodes $V$, i.e., $Z$ is a family of random variables $\{ Z_v:v\in V\}$ such that $Z_v: \Omega\rightarrow S$ is measurable for each $v\in V$. We denote the indicator function by $\mathbbm{1}$ and we define the $\alpha$-mixing coefficient of the random field $\{Z_v: v\in V\}$ on the graph $G=(V,E)$ by
\begin{align*}
	\alpha_G(n) \coloneqq \sup_{ \substack{ I, J \subseteq V,\\ d_G(I,J) \ge n}} \sup_{ \substack{ A \in \cF( I ),\\ B\in \cF(J)} } | \p(A \cap B) - \p(A)\p(B) |, \quad n\in\N.
	\end{align*}
The random field is strong mixing w.r.t.\ $G$ if and only if $\alpha_G(n) \rightarrow 0$ for $n\rightarrow \infty$. In the sequel, we investigate random fields which are defined on the following class of graphs:

\begin{definition}[Trees growing at an exponential rate $A$]
Let $A\in\N_+$. A tree $T=(V,E)$ is growing at an exponential rate $A$ if $T$ is a rooted tree and each node $v \in V$ has exactly $A$ children. The nodes in the tree are labeled according to the following scheme: the distinguished root (which has no parent) is labeled by $(0,0)$ and the children of the node $(j,k)$ are labeled by $(j+1, A (k-1) + 1),\ldots, (j+1, A k)$. Hence, the set of nodes and the set of edges are given by
\[
	V = \left\{ (j,k): j\in\N,\, 1\le k \le A^j \right\} \text{ and } E = \left\{ ( (j,k), (j+1,k') ):\, (j,k)\in V \text{ and } A(k-1)+1 \le k' \le A k \right\}.
	\]
A rooted graph $G=(V,E)$ is growing at an exponential rate $A$ if the edges $E$ can be decomposed into two disjoint sets as $ E = E' \cup \tilde{E}$ such that $(V,E')$ is a tree growing at an exponential rate $A$ and the set $\tilde{E}$ of additional edges has the property that it does not connect nodes of arbitrary length in $T$, i.e.,
\[
		\sup\{ d_T (v,w) \,|\, (v,w)\in \tilde{E} \} < \infty.
	\]
\end{definition}

We come to the definition of a mixing embedding of a graph. Here it is worthy to mention that $-$ especially in the context of graph theory $-$ there are different definitions of graph embeddings: the common definition of an embedding of a graph $G$ requires, loosely speaking, that the edges of the embedded graph may only intersect at their endpoints, i.e., at the nodes. It is well known that any graph with countably many nodes can be embedded into $\Z^3$ via placing the $i$-th node at the point $(i, i^2, i^3)\in\Z^3$, compare \cite{cohen1994three}. Furthermore, one can characterize the finite graphs which are embeddable into the plane (the planar graphs) with the help of the theorems of \cite{kuratowski1930probleme} and of \cite{wagner1937eigenschaft}. Here, we slightly change this graph theoretic definition such that it is tailored to our needs: we can omit the restriction that edges may not intersect at an interior point. However, since we shall usually be dealing with infinite graphs, we have to add a requirement that is essential when is comes to mixing random fields which are defined on the graph which is to be embedded.  We need this definition to show what is intuitively clear: the Bernstein inequalities for regular lattices are not applicable in the context of graphs which grow at an exponential rate. We give the definition

\begin{definition}[Mixing embedding of a graph]\label{embeddableGraph}
Let $G=(V,E)$ be a graph with countably many nodes $V$ and denote by $d_{p,\Z^N}$ the Euclidean $p$-norm on the $N$-dimensional lattice $\Z^N$, for $p\ge 1$ and $N \in \N_+$. There is a mixing embedding of $G$ in $\Z^N$ if there is a dimension $N \in \N_+$ such that $G$ is isomorphic to a graph $G'=(V',E')$ with $V' \subseteq \Z^N$ and for each sequence $( (v_i, w_i): i\in \N ) \subseteq V \times V$ with image $( (v'_i, w'_i): i\in \N) \subseteq V' \times V' $ it is true that
\[
		\sup \{ d_G(v_i,w_i): i\in \N \} < \infty \Longrightarrow 	\sup \{ d_{\infty,\Z^N} (v'_i,w'_i): i\in \N \} < \infty.
\]
\end{definition}

In the following, when speaking of the lattice $\Z^N$ as a graph, we shall always understand the graph $G=(V,E)$ with nodes $V = \Z^N$ and edges $E = \{ (v,v+b_i): v \in V, \; i=1,\ldots,N \}$ where $b_i$ is the $i$-th standard basis vector which is one in the $i$-th coordinate and zero otherwise. Note that in this case, we have $d_G \equiv d_{1,\Z^N}$ and $d_{\infty,\Z^N} \le d_{1,\Z^N} \le N d_{\infty,\Z^N}$. We have a practical lemma which gives equivalent formulations of this definition

\begin{lemma}\label{EquivCondEmb}
Let $G$ be a graph. Then the following are equivalent
\begin{enumerate}
\item There is a mixing embedding of $G$ in $\Z^N$
\item $G$ is isomorphic to a graph $G'=(V',E')$ with nodes $V' \subseteq \Z^N$ and there is a constant $0< C < \infty$ such that for any $(v, w) \in V \times V$ with image $(v', w')\in V' \times V'$ it is true that $d_{\infty,\Z^N} (v',w') \le C\, d_G (v,w)$.
\item $G$ is isomorphic to a graph $G'=(V',E')$ with nodes $V' \subseteq \Z^N$ and 
$$\sup \{ d_{ \infty, \Z^N} (v', w') :\, v \in V, w\in \cN(v), v \text{ (resp. } w) \text{ is isomorphic to } v' \text{ (resp.} w') \} < \infty. $$
\end{enumerate}
In particular, let $\{Z_v: v\in V\}$ be a random field on $G$, denote by $\{Z_s: s\in V'\}$ the same random field under the graph isomorphism. Then the mixing coefficients satisfy asymptotically	$\alpha_{\infty,\Z^N}	(\ceil*{C\,\cdot\,}) \le \alpha_{G}$ which means that strong mixing is inherited when switching between $G$ and $G'$.
\end{lemma}
\begin{proof}
(1) $\Rightarrow$ (2) and (3): assume that there is a mixing embedding of $G$ in $\Z^N$, then obviously $V$ is countable, thus, the number
\[
		C \coloneqq \sup \{ d_{\infty,\Z^N}(v,w): d_G(v,w) = 1, v,w \in V \} 
\]
is meaningful and finite by assumption. Consequently, we have for two connected nodes $v$ and $w$ in $V$ that $d_{\infty,\Z^N}(v,w) \le C \, d_G(v,w)$. If $v$ and $w$ are not connected then $d_G(v,w) = \infty$. Hence, $C$ is the proper constant. The converse inclusions (2) (resp. (3)) $\Rightarrow$ (1) are immediate.\\
We come to the amendment of the lemma. Let $n\in \N$ be given and consider a random field on $G$ and its graph-isomorphic counterpart on $G'$. We infer for two sets $I',J' \subseteq V'$ with preimage $I,J \subseteq V$ and $d_{\infty,\Z^N} (I',J') \ge n$ that $C\, d_G( I, J) \ge d_{\infty,\Z^N}(I',J') \ge n$, i.e., we have using the graph isomorphism for $n\ge C$
\[
		\{ (I',J'):\, I',J' \subseteq V' \text{ and } d_{\infty,\Z^N}(I',J') \ge n \} \subseteq \{ (I,J):\, I,J \subseteq V \text{ and } d_G(I,J) \ge C^{-1}\,n \}.
\]
Thus, $\alpha_{\infty,\Z^N}(n) \le \alpha_G \left( \floor*{C^{-1}\,n} \right)$ for $n\ge C$. This means that asymptotically $\alpha_{\infty,\Z^N}\le \alpha_G\left(\floor*{C^{-1}\,\cdot\,} \right)$ or rather $\alpha_{\infty,\Z^N}\left(\ceil*{C\,\cdot\,} \right)\le \alpha_G$. 
\end{proof}

The following class of graphs does not allow for a mixing embedding in $\Z^N$
\begin{proposition}
Let $G=(V,E)$ be a graph with root $v_0 \in V$. Put $L_0 \coloneqq \{v_0\}$ and recursively
\[
	L_k \coloneqq \left\{ v \in V\setminus \cup_{i=0}^{k-1} L_i \;|\; \exists w \in L_{k-1} \text{ with } w \in \cN(v) \right\}
	\]
for $k\in \N_+$. If the map $\N \ni k \mapsto |L_k|$ grows faster than any polynomial function of degree $N$ defined on $\N$, there is no mixing embedding of $G$ in $\Z^N$.
\end{proposition}     
\begin{proof}
Let the map $\N\ni k \mapsto |L_k|$ grow faster than any polynomial of degree $N$ and assume that there is a mixing embedding of $G$ in $\Z^N$ for some $0<C<\infty$ which satisfies $d_{\infty,\Z^N}(v,w) \le C d_G(v,w)$ as stated in Lemma~\ref{EquivCondEmb}. First, observe that for $v$ and $w$ both in $L_k$ the distance in the graph is at most $d_G(v,w)\le 2k$.  By assumption there is a $k_1 \in \N$ such that for all $k\ge k_1$ we have $|L_k| > (2k+1)^N$. Thus, for $k\ge k_1$ there are $v,w\in L_k$ with the property that $d_{\infty,\Z^N}(v,w) > 2k$ which implies for these two nodes that
\[
			2k < d_{\infty,\Z^N}(v,w) \le C d_G(v,w) \le 2 C k.
\]
Hence, $C>1$. In the same way, there is a $k_2 \in \N$ such that for all $k \ge k_2$, we have $|L_k| > (2C^2 k+1)^N$. In particular, there are $(v,w) \in L_k$, $k\ge k_2$ with the property that $d_{\infty,\Z^N} (v,w) > 2C^2 k$ which implies for $k \ge \max( k_1, k_2)$
\[
	2C^2k < d_{\infty,\Z^N}(v,w) \le C d_G(v,w) \le 2  C k;
\]
which in turn implies $C <1$. This contradicts the assumption that there is a mixing embedding of $G$ in $\Z^N$.
\end{proof}

This implies that we cannot use the above mentioned Bernstein inequalities for data which is defined on a lattice to derive concentration inequalities for random fields that are defined on graphs which grow at an exponential rate $A$. Instead we give a new Bernstein inequality which can deal with this class of random fields in the next section.

\section{A Bernstein inequality for exponentially growing graphs}\label{Bernstein}
In this section we derive inequalities of the Bernstein type for random fields which are highly-connected and whose index set grows at an exponential rate. We need the following important lemma:

\begin{lemma}\label{TreeNumberOfPairs}
Let $T=(V,E)$ be a tree growing at an exponential rate $A$. Denote by
\begin{align}\label{EqTreeNumberOfPairs0a}
		V(j,k,P) \coloneqq \left\{	(j',k')\in V|\, j\le j' \le j+P-1,\, A^{j'-j}(k-1) + 1 \le k' \le A^{j'-j} k		\right\} 
\end{align}
the set of nodes of the subtree of $T$ which has its root at the node $(j,k)$ and consists of $P \in \N_+$ generations. Consider the graph which is induced by the set of nodes $V(j,k,P)$. Then the number of pairs $(v,w)$ in this graph which are separated by exactly $L$ edges for $1\le L \le 2(P-1)$ is given by
\begin{align} 
	N(P,L)\coloneqq	&\sum_{h = 0}^{ \floor*{P-1-L/2}} A^h \sum_{i = 1 \vee ( L - (P-1-h) ) }^{L \wedge (P-1-h) } A^{i} \left( 	2\cdot 1_{\{ L = i\} } + (A-1) A^{L-i-1} \1{ L > i  }	\right) \label{EqTreeNumberOfPairs0b}\\
		&= 2 \frac{ A^P - A^L}{A-1} 1_{\{ L \le P-1\} } + (L-1) \left( A^P - A^{L-1} \right) \1{ L \le P  }  \nonumber\\
		&\quad +  \big(	2(P-1) - L + 1 \big) \left( A^{P-1 + \floor*{L/2}} - A^{ L-1 \vee P } \right) \,\cdot\, \1{ L \ge 4	} \nonumber\\
		&\quad  - 2 \frac{ \big\{ \, (A-1) (P-1-\ceil*{L/2} ) - 1 \, \big \}\, A^{ P - 1 + \floor*{L/2}} + A^{L} }{A - 1} \,\cdot\, \1{ L \ge 4 } \nonumber \\
		&\quad  + 2 \frac{ \left\{ (A-1)(P-L) -1	\right\}A^P + A^L}{A-1} \1{4 \le L < P } \nonumber \\
		&\le C\,P\,A^{P + L/2}, \nonumber
\end{align}
for a suitable constant $0 < C <\infty$ which does not depend on $P$, $L$ and $A$.
\end{lemma}
\begin{proof}[Proof of Lemma~\ref{TreeNumberOfPairs}]
The minimal distance in this subtree clearly is 1, whereas the maximal distance is $2(P-1)$. Let now a length $L$ be fixed, $1\le L \le 2(P-1)$. We distinguish two cases for a pair $(v,w)$ which is separated by $L$ edges: in the first case $w$ (resp. $v$) is a descendant of $v$ (resp. $w$). In the second case $v$ and $w$ have a common parent which we call $r$ and, plainly, $v\neq r \neq w$.\\
The first case is only possible for $1 \le L \le P-1$, for such an $L$ there are exactly $2 ( A^P-A^L )/(A-1)$ such pairs $(v,w)$ in this subtree. The second case is possible for $2 \le L \le 2(P-1)$. Depending on $L$ the parent is located between generation zero and generation $\ceil*{ P-1-L/2}$, denote its generation by $h$. Having fixed a parent $r$ in generation $h$ the distance from $r$ to the first node $v$ is at least $1 \vee (l-(P-1-h))$ and at most $L \wedge (P-1-h)$, denote this distance by $i$. Hence, there are exactly $A^{i}$ nodes in question for $v$. In this case that $i < L$ the node $w$ is separated $L-i$ generations from $r$. Since $v\neq w$ and their graph distance is $L$, this yields $(A-1) A^{L-i-1}$ possibilities for $w$. All in all, we give the number of pairs with the formula from equation~\eqref{EqTreeNumberOfPairs0b}.
\end{proof}

It follows the Bernstein inequality. Here we do not consider the full set of nodes $V$ instead we focus on a strip of $V$ which is defined with the help of the $V(j,k,P)$ from the previous Lemma~\ref{TreeNumberOfPairs}.
\begin{theorem}[Bernstein inequality]\label{BernsteinTreeStrip}
Let $T=(V,E)$ be a tree growing at an exponential rate $A$. Let $Z_v$ be a real-valued random variable for each $v\in V$  with $\E{Z_v} = 0$, $\norm{Z_v}_{\infty} \le C$ and $\text{Var}(Z_v) \le \sigma^2$, for some $ 0 < \sigma, C < \infty$. Let $L\in \N$, $P \in \N_+$ and consider the subtree induced by the set of nodes
\begin{align}\label{EqBernsteinTreeStrip0a}
		V' \coloneqq V(L,1,P) \cup \ldots \cup V(L,A^L,P)
\end{align}
with $V(L,i,P)$ as in the definition given in \eqref{EqTreeNumberOfPairs0a}. Then
\begin{align}\begin{split} \label{EqBernsteinTreeStrip0b}
		\p\left( \left| \sum_{v \in V'} Z_v \right| > \epsilon  \right) 
		& \le 2 e^{-\beta\epsilon} \exp\left\{ 10 \sqrt{e} \alpha_T( f)^{ (P_2 + Q_2) / \left(2P_2 + 2Q_2 + A^L \right) } \frac{A^L}{P_2 + Q_2 } \right\}  \\
		&\quad \cdot \exp\left\{ 4 \beta^2 e\, (P_2)^2 \left(\frac{ A^P-1}{A-1} \sigma^2 + 4 C^2  \sum_{k=1}^{2(P-1)} \alpha_T(k)\, N(P,k) \right) \left( \frac{A^L}{P_2 + Q_2} + 1		\right) \right\}
\end{split} \end{align}
where $Q_2, P_2 \in \N_+$ such that $Q_2 \le P_2$ and $P_2 + Q_2 <  A^L$ as well as
 $$\beta \le \frac{A-1}{4e C P_2 (A^P-1) } \text{ and } f \coloneqq 2 \ceil*{ \frac{ \log Q_2}{\log A }}. $$
\end{theorem}
\begin{proof}[Proof of Theorem~\ref{BernsteinTreeStrip}]
We have to partition $V'$ suitably. We use the abbreviations $\widetilde{V}(\,\cdot\,) \coloneqq V(L ,\,\cdot\, ,P)$ and $T \coloneqq \ceil*{  A^L /(P_2 + Q_2) }$ as well as,
\begin{align*}
		A(i) &\coloneqq \widetilde{V}\big( (i-1) (P_2 + Q_2) + 1\big) \cup \ldots \cup \widetilde{V}\big( (i-1)Q_2 + i P_2 \big) \\
		B(i) &\coloneqq \widetilde{V}\big( (i-1) Q_2 + i P_2 + 1 \big) \cup \ldots \cup \widetilde{V}\big( i (P_2 + Q_2) \big),
\end{align*}
for $i=1,\ldots,T$. Note that the $A(i)$ and $B(i)$ are the union of the disjoint sets $\widetilde{V}(\,\cdot\,)$ and that some $A(i)$ and $B(i)$ might be empty. Furthermore, we define
\[
	V'_1 \coloneqq \cup_{i=1}^T A(i) \text{ and } V'_2 \coloneqq \cup_{i=1}^T B(i).
	\]
Then, we have with Markov's inequality and the well-known AM-GM inequality that
 \[ \p\left( \sum_{v\in V'} Z_v > \epsilon \right)  \le \frac{e^{-\beta \epsilon} }{2} \left\{ \E{ e^{2\beta \sum_{v\in V'_1} Z_v } } + \E{ e^{ 2\beta \sum_{v\in V'_2 } Z_v }} \right\} \text{ for } \beta > 0. \]
Hence, it suffices to consider the sum $\sum_{v\in V'_1} Z_v$ closer. We write 
\[
		S(i) \coloneqq \sum_{v \in \cup_{i=1}^i A(i) } Z_v \text{ and } J(i) \coloneqq \sum_{v\in A(i) } Z_v  \text{ for } i = 1,\ldots,T.
\]
We compute the expectations of the random variables $e^{\delta S(i)}$, for $\delta > 0 $ sufficiently small. Note that the distance w.r.t.\ $d_G$ between $v \in A(i)$ and $v' \in A(i')$, $i\neq i'$, is at least $2\ceil*{\log Q_2 / \log A }$.  Since $S(i) = S(i-1) + J(i)$, we infer from Davydov's inequality given in Proposition~\ref{davydovsIneq} that
\begin{align}
	\E{ e^{\delta S(i) } } &= \text{Cov}( e^{\delta S(i-1)}, e^{\delta J(i) } ) + \E{ e^{\delta S(i-1)}} \E{ e^{\delta J(i) }} \nonumber \\
	& \le 10\, \alpha_T( f)^{1/a} \norm{ \exp( \delta S(i-1) ) }_{b} \norm{ \exp( \delta J(i) ) }_{\infty } + \E{ e^{\delta S(i-1)}} \E{ e^{\delta J(i) }}  \label{EqBernstein1}
\end{align}
for H{\"o}lder conjugate $a,b \ge 1$ and $f \coloneqq 2 \ceil*{  \log Q_2/ \log A}$. Furthermore, we have if $| \delta J(i)| \le \ 1/(2e)$ that 
$$\exp \delta J(i) \le 1 + \delta J(i) + \delta^2 J(i)^2.$$
Now the random variables $Z_v$ are essentially bounded by $C$. Let $\beta \le (A-1)/(4eC P_2 (A^P-1))$ and define $\delta \coloneqq 2\beta$. Then, we have
\[
		| \delta J(i)| \le \delta C P_2 \, \frac{A^P-1}{A-1} \le \frac{1}{2e}, \text{ hence, } \E{ \delta J(i) } \le \exp\left( \delta^2 \E{J(i)^2} \right).
\]
Note that in the subgraph induced by the $A(i)$ there are exactly $N(P,k)$ pairs of nodes $(v,w)$ with $d_G( v,w) = k \in \{1,\ldots,2(P-1) \}$, where $N(P,k)$ is given in Lemma \ref{TreeNumberOfPairs}. For the next two lines we use the inequality $\left(\sum_{i=1}^n a_i \right)^2 \le n^2 \sum_{i=1}^n a_i^2 $ for real numbers $a_i$, $i=1,\ldots,n$, $n\in\N$. Consequently, we get
\begin{align*}
		\E{ J(i)^2 } &= \sum_{v \in A(i)} \E{ Z_v^2 } + \sum_{ \substack{v,w\in A(i), \\ v\neq w} } \E{ Z_v Z_w } \\
		&\le (P_2)^2 \left\{	\frac{ A^P-1}{A-1} \sigma^2 + 4 C^2  \sum_{k=1}^{2(P-1)} \alpha_T(k) \,N(P,k)		\right\} =: K
\end{align*}
with Davydov's inequality from Proposition~\ref{davydovsIneq}.\\
Furthermore, we find with the H{\"o}lder inequality that $\norm{ \exp(\delta S(i-1) ) }_1 \le \norm{ \exp(\delta S(i-1) ) }_b$. Thus, equation \eqref{EqBernstein1} can be bounded by
\begin{align}
		\E{ \exp\left(\delta S(i) \right)} \le \left\{ 10\, \alpha_T(f)^{1/a} \norm{ \exp(\delta J(i) }_{\infty} + \exp( \delta^2 K  ) \right\} \norm{ \exp(\delta S(i-1) ) }_b. \label{EqBernstein2}
\end{align}
Especially, for the case $i=T$ successive iteration of \eqref{EqBernstein2} yields for the choice $a \coloneqq T+1$ and $b = 1+1/T$ (as in \cite{frankeBernstein})
\begin{align*}
	\E{ \exp\left( \delta S(T) \right) } \le \exp\{	10 \sqrt{e} \alpha_T(f)^{1/(T+1)} (T-1) + \delta^2 e K T \}.
\end{align*}
Next, since $\alpha_T(f) \le 1$ and $1/(1+T) \ge \frac{P_2+ Q_2}{2(P_2 + Q_2) + A^L}$, we arrive at
\begin{align*}
	\E{ \exp\left( 2\beta \sum_{ v\in V'_1} Z_v		\right) } &\le \exp \left\{	10\sqrt{e} \alpha_T(f)^{ (P_2 + Q_2)/(2(P_2+Q_2) + A^L)} \frac{ A^L}{ P_2 + Q_2}	\right\} \\
	&\quad \cdot \exp\left\{ 4 \beta^2 e (P_2)^2 \left[	\frac{A^P-1}{A-1} \sigma^2 + 4 C^2 \sum_{k=1}^{2(P-1)} \alpha_T(k) \, N(P,k)	\right]\left( \frac{A^L}{P_2  + Q_2} +1 \right) 	\right\}.
\end{align*}
The computations for $\E{ \exp\left( 2\beta \sum_{ v\in V'_2} Z_v		\right) }$ are similar and one achieves the same bounds for this term. This finishes the proof.
\end{proof}

We are now in position to derive a concentration inequality. We consider an infinite tree which grows at an exponential rate $A$ and which is endowed with a random field $Z$. We assume that the random field $Z$ on the tree $T$ is strong mixing such that 
\begin{align}\label{EqSummableMixing}
		\sum_{k\in \N} \alpha_T(k) N(P,k) \in \cO\left(	P A^P	 \right),
\end{align}
where $N(P,k)$ is defined in Lemma~\ref{TreeNumberOfPairs}. We say that the mixing coefficients decay at a \textit{super-exponential} (or \textit{hyper-exponential}) rate if there is a positive increasing function $g$ with $\lim_{n \rightarrow \infty} g(n) = \infty$ such that
\begin{align}
		\alpha_T( n) \le \exp( - n g(n) ). \label{EqSuperExpMixing}
\end{align}
In this case, equation~\eqref{EqSummableMixing} follows from Lemma~\ref{TreeNumberOfPairs} with the bound $N(P,k) \in \cO\left( P A^{P+k/2}\right)$ and the following concentration inequality is true

\begin{theorem}[Concentration inequality for exponentially growing trees]\label{ConcentrationTree}
Let $T=(V,E)$ be a tree growing at an exponential rate $A$ and let $Z$ be a random field on $T$ as in Theorem~\ref{BernsteinTreeStrip}. Let the random field be strong mixing w.r.t.\ the graph metric with $\alpha$-mixing coefficients which fulfill \eqref{EqSummableMixing}, e.g., the mixing coefficients decay at a super-exponential rate as in \eqref{EqSuperExpMixing}. Consider the subgraph which consists of the first $L$ generations of $T$ for $L\in \N$ 
\[
		V_L = \left\{ (j,k) \in V: 0 \le j \le L-1,\, 1 \le k \le A^{L-1} \right \}.
\]
Then there are constants $c_1, c_2 \in \R_+$ such that for all $L\in\N$ and $\epsilon > 0$
\begin{align*}
		\p\left( \frac{1}{|V_L|} \left|\sum_{v\in V_L} Z_v \right| > \epsilon	\right) \le c_1 \exp\left\{	- c_2\, \epsilon \frac{L}{\log L} \right\}.
\end{align*}
This means the probability decays asymptotically at a rate which is approximately linear in the size of the sample $V_L$.
\end{theorem}
\begin{proof}[Proof of Theorem~\ref{ConcentrationTree}]
Let $P_1 := \floor*{L^{\eta}}$ for some $\eta\in (0,1)$. We partition $V_L$ in the following way: first we define the wedge which consist of the first $L-P_1$ generations
$$	W_L \coloneqq \{ (j,k)\in V_L:\, 0\le j < L - P_1 \}. $$
The remaining $P_1$ generations are collected in
$$ U_L \coloneqq \{(j,k) \in V_L:\, L-P_1 \le j < L \}.$$
The sums which correspond to these partitioning are $\widetilde{S}_L \coloneqq\sum_{v \in W_L} Z_v$ and $S_L \coloneqq \sum_{v \in U_L} Z_v$. Then we split the probability as follows,
\begin{align}
		\p\left( \left|\sum_{v\in V_L} Z_v \right| > \epsilon\, |V_L|		\right) & \le \p\left( |\widetilde{S}_L |	 > \frac{\epsilon}{2} |V_L|	\right) + \p\left( |S_L|	 > \frac{\epsilon}{2} |V_L|	\right) \label{EqConcentrationTree1} 
\end{align}
The first probability in \eqref{EqConcentrationTree1} is negligible because we find
\begin{align*}
	\p\left( |\widetilde{S}_L| > \frac{\epsilon}{2} |V_L|  \right) &\le \1{ 2C \frac{ A^{L- P_1}-1}{A-1} > \frac{\epsilon}{2} \frac{ A^L-1}{A-1} }.
\end{align*}
Thus, we can focus on the second probability in \eqref{EqConcentrationTree1}. We use Theorem~\ref{BernsteinTreeStrip}. We make the following definitions
\begin{align*}
		P_2 &\coloneqq Q_2\coloneqq\floor*{ D \frac{ A^{L-P_1} }{L-P_1} \log L }, \text{ for a sufficiently large constant $D\in\R_+$}\\
		\beta &\coloneqq \frac{ A-1 }{ 4eC P_2 (A^{P_1}-1) } \text{ and } f \coloneqq 2\ceil*{ \frac{ \log P_2}{\log A} }.
\end{align*}
Consider the exponent of the first factor given in \eqref{EqBernsteinTreeStrip0b}: one finds that there is a constant $c\in\R_+$ which does neither depend on $L$ nor on $\epsilon$ nor on the $Z_v$ such that
\begin{align}
		 \exp\left\{	-\frac{\epsilon}{2} \beta\, |V_L|  \right\} &\le \exp\left\{-c \, \epsilon \frac{L}{\log L }	\right\}. \label{EqConcentrationTree2}
\end{align}
The second factor in \eqref{EqBernsteinTreeStrip0b} is given by
\begin{align}\label{EqConcentrationTree3}
		\exp \left\{	B\sqrt{e} \alpha_T( f )^{2P_{2}/(4 P_{2} + A^{L-P_1} ) } \frac{ A^{L-P_1}}{2 P_{2}}	\right\}.
\end{align}
We can derive the following bound for the mixing coefficient and the exponent inside the $\exp$-function in \eqref{EqConcentrationTree3}
\begin{align*}
	\alpha_T(f)^{2 P_{2}/(4 P_{2} + A^{L-P_1} ) } &\le \exp\left\{ - \frac{D}{5} \log L \left( 1 + \frac{\log( D \log L /(L-P_1) ) }{ (L-P_1)\log A } \right)\right\}.
	\end{align*}
Consider the second factor inside the $\exp$-function in \eqref{EqConcentrationTree3}, it is $ A^{L-P_1}/P_2 \le (L-P_1) / \log L$. In particular, the second factor in \eqref{EqConcentrationTree3} is uniformly bounded for all $L\in \N_+$ if $D$ is sufficiently large. Consider the third factor in \eqref{EqBernsteinTreeStrip0b}. Since the mixing coefficients decay sufficiently fast, we can derive the following inequality 
\begin{align*}
		&\exp\left\{ 4e \beta^2 (P_{2})^2 \left(	\frac{A^{P_1}-1}{A-1}\sigma^2 + 4C^2 \sum_{k=1}^{2(P_1-1)} N(P_1,k) \alpha_T(k)	\right) \left( \frac{A^{L-P_1}}{ 2 P_2 } + 1 \right)\right\}  \le \exp\left\{ c \frac{ P_1 (L-P_1) }{A^{P_1} \log L }	\right\},
\end{align*}
for a suitable constant $c\in\R$. In particular, this expression is uniformly bounded over all $L\in\N$. All in all, we have shown that there are constants $c_1, c_2 \in \R_+$ such that for the second probability in \eqref{EqConcentrationTree1} is bounded as
\[
		\p\left( |S_L|	 > \frac{\epsilon}{2} |V_L|	\right) \le c_1 \p\left(	- c_2\, \epsilon \frac{L}{\log L} \right),
\]
where, the asymptotic speed is determined by \eqref{EqConcentrationTree2}. This completes the proof.
\end{proof}

The previous theorem can be applied to exponentially growing graphs as well, we have the useful corollary:
\begin{corollary}[Concentration inequality for exponentially growing graphs]\label{ConcentrationGraph}
Let $G=(V,E' \cup \tilde{E} )$ be a graph growing at an exponential rate $A$ endowed with a random field $Z$ as in Theorem~\ref{ConcentrationTree}. Then there are constants $c_1, c_2 \in \R_+$ such that for all $L\in\N$ and $\epsilon > 0$
\begin{align*}
		\p\left( \frac{1}{|V_L|} \left|\sum_{v\in V_L} Z_v \right| > \epsilon	\right) \le c_1 \exp\left\{	- c_2\, \epsilon \frac{L}{\log L} \right\}.
\end{align*}
\end{corollary}
\begin{proof}[Proof of Corollary~\ref{ConcentrationGraph}]
We only need to show that the mixing conditions for the tree $T=(V,E')$ are fulfilled. The condition that $S\coloneqq \sup \{ d_T(v,w) \,|\, (v,w)\in \tilde{E} \} < \infty$ implies that 
$$ 1 \vee \frac{ d_T(v,w)}{S}  \le d_G(v,w) \le d_T(v,w) .$$
In particular, the mixing rates w.r.t.\ the tree and the whole graph structure satisfy asymptotically the inequality relations $\alpha_T( \ceil*{S\cdot} ) \le \alpha_G \le \alpha_T$.
Thus, we can conclude the statement from Theorem~\ref{ConcentrationTree}.
\end{proof}

\appendix
\section{Appendix}\label{Appendix}

\begin{proposition}[\cite{davydov1968convergence}]\label{davydovsIneq}
Let $\pspace$ be a probability space and let $\cG, \cH \subseteq \cA$ be sub-$\sigma$-algebras. Denote by $\alpha \coloneqq \sup\{ |\p(A \cap B) - \p(A)\p(B) |:\, A \in \cG, B\in \cG \}$ the $\alpha$-mixing coefficient between $\cG$ and $\cH$. Let $p,q,r \ge 1$ be H{\"o}lder conjugate. Let $\xi$ (resp. $\eta$) be in $L^p(\p)$ and $\cG$-measurable (resp. in $L^q(\p)$ and $\cH$-measurable). Then
\[
	\left| \text{Cov}(\xi,\eta) \right|  \le 10\, \alpha^{1/r} \norm{\xi}_{L^p(\p)} \norm{\eta}_{L^q(\p)}
\]
\end{proposition}




\end{document}